\documentclass{amsproc}

\usepackage{amssymb}

\newtheorem{thm}{Theorem}[section]
\newtheorem{lem}[thm]{Lemma}
\newtheorem{prop}[thm]{Proposition}
\newtheorem{cor}[thm]{Corollary}
\newtheorem{exs}[thm]{Examples}
\newtheorem{ex}[thm]{Example}

\theoremstyle{definition}

\theoremstyle{remark}
\newtheorem{remark}[thm]{Remark}

\numberwithin{equation}{section}

\newcommand{\sbq}{\subseteq}

\newcommand{\mc}{\mathcal}
\newcommand{\vide}{\emptyset} 
\newcommand{\tbf}{\textbf}
\newcommand{\mbf}{\mathbf}
\newcommand{\mf}{\mathfrak}
\newcommand{\qcl}{\ensuremath{Q_{\text{cl}}\,}}
\newcommand{\inv}{^{-1}}
\newcommand{\bz}{\mbf{0}}

\newcommand{\C}{\text{C}}
\DeclareMathOperator{\spec}{Spec}

\DeclareMathOperator{\ann}{ann}

\DeclareMathOperator{\coz}{coz}
\DeclareMathOperator{\z}{\mbf{z}}
\DeclareMathOperator{\supp}{supp}

\newcommand{\R}{\mathbf R}
\newcommand{\N}{\mathbf N}
\newcommand{\Q}{\mathbf Q}
\newcommand{\Z}{\mathbf Z}
\newcommand{\B}{\mathbf B}

\DeclareMathOperator{\cl}{cl}
\newcommand{\varep}{\varepsilon}
\newcommand{\res}{\raisebox{-.5ex}{$|$}}
\newcommand{\ov}{\overline}

\newcommand{\cx}{\text{C}(X)}

\newcommand{\rr}{\tbf{rr}}

\newcommand{\rro}{\le_{\text{rr}}}
\newcommand{\wrr}{\wedge_{\text{rr}}}

\newcommand{\radi}{\sqrt{I}}

\begin{document}

\title[Reduced ring order II]{The reduced ring order and lower semi-lattices II\\ \emph{preprint}}

 \author{W.D. Burgess}
\address{Department of Mathematics and Statistics\\ University of Ottawa, Ottawa, Canada, K1N 6N5}
\curraddr{}
\email{wburgess@uottawa.ca}

\author{R. Raphael}
\address{Department of Mathematics and Statistics\\ Concordia University, Montr\'eal, Canada, H4B 1R6}
\curraddr{}
\email{r.raphael@concordia.ca}

\subjclass[2010]{06F25\ 16W80 46E25\ 13B30}

\date{}\keywords{reduced ring order, lifting orthogonal sets, localizations, generalized power series rings}

\begin{abstract} This paper continues the study of the reduced ring order (\rr-order) in reduced rings where $a\rro b$ if $a^2=ab$. A reduced ring is called \rr-good if it is a lower semi-lattice in the order. Examples include weakly Baer rings (wB or PP-rings) but many more.  Localizations are examined relating to this order as well as the Pierce sheaf. Liftings of rr-orthogonal sets over surjections of reduced rings are studied. A known result about commutative power series rings over wB rings is extended, via methods developed here, to very general, not necessarily commutative, power series rings defined by an ordered monoid, showing that they are wB. 
\end{abstract}
\maketitle

\setcounter{section}{1}  \setcounter{thm}{0}
 \tbf{1. Introduction.}   The \emph{reduced ring order}, here called the \emph{\rr-order}, is defined on a reduced ring  $R$ (a ring without non-trivial nilpotent elements) by $a\rro b$ if $a^2=ab$. This order has also been called Abian's order and Conrad's order.  The earliest reference to this order seems to be \cite{S}. The order, appropriately modified, has also appeared in the study of certain semigroups and more general semiprime rings (\cite{BR2}).  This is a partial order that generalizes the natural order on boolean rings.  It has been used to characterize direct products (\cite{Ch}) and to look at the complete ring of quotients of a commutative reduced  ring as an orthogonal completion in this order (see, for example \cite{BR1}).  More recently, in \cite{BR3}, the question of when a reduced ring is a lower semi-lattice in this order was examined as well as when countable orthogonal sets (in the \rr-order) can be lifted over a ring surjection. 
 
 For a pair of elements $a,b\in R$, it is rare for an \rr-supremum to exist (\cite[Propositions~2.7 and 2.8]{BR3}) but the analogy with boolean rings is closer when each pair has an \rr-infimum making $R$ a lower semi-lattice in the \rr-order.  When this occurs for a ring $R$, the ring is called \emph{\rr-good}.

In some ways this article can be considered as a continuation of \cite{BR3} but it looks at different aspects of the \rr-order.

The paper is divided into sections. Section~2 develops some tools used later in the article and also studies connections with integral extensions. Section~3 takes another look at the question of lifting \rr-orthogonal sets over a surjective ring homomorphism of reduced rings $R\to S$, a major topic from \cite{BR3}. When the ring $R$ is \rr-good then an \rr-orthogonal pair in $S$ can always be lifted to an \rr-orthogonal pair in $R$.  This is not the case if $R$ does not have this property, as an example shows. Lifting countable \rr-orthogonal sets is revisited at the end of the section. 
 
In Section~4, the behaviour of the \rr-good property in localization in (mostly) commutative rings is examined. The results are mixed. I.e., there are cases where the \rr-good property is preserved and cases where it is not.  An example shows that a reduced local ring need not be \rr-good. 
 
Section~5 looks at the \rr-good property in terms of the Pierce sheaf.  Having Pierce stalks that are \rr-good is a necessary but not a sufficient condition for a ring to be \rr-good, as an example shows. Weakly Baer rings (also known as pp-rings) are \rr-good and the property of being weakly Baer passes to very general forms of power series rings, commutative or not, generalizing known results in the commutative case. This is the content of Section~6 which uses methods developed in Section~5.
  
 A companion article, \cite{BR4}, looks at the \rr-order in the special case where the ring is of the form $\cx$, the ring of real valued continuous functions on a topological space $X$.  The goal there is to connect topological properties of the space $X$ and the behaviour of $\cx$ with respect to the \rr-order, a topic first looked at in a section of \cite{BR3}.  Some of the key examples in the present article are of the form $\cx$. 
 
Before moving to details, here is a list of notations.\\[.5ex]
 \noindent\tbf{Terminology and definitions.} \\[-1.7ex]
 
(1) All rings considered in the sequel will be \emph{reduced} (with no non-trivial nilpotent elements) and unitary. 

(2) In a ring $R$, the annihilator of $a\in R$ is denoted $\ann a$; note that in a reduced ring left and right annihilators coincide. 

(3) The \rr-order in $R$ is defined by $a\rro b$ if $a^2=ab$. This is a partial order on the set $R$. It is preserved under homomorphisms between reduced rings. Notice that in $R$, non-zero divisors are maximal in the \rr-order and that 0 is the unique minimal element. 

(4) If in $R$, $c\rro a$ and $c\rro b$ then the notation $c\rro a,b$ is used. 

(6) For two elements $a,b\in R$, their \emph{infimum} in the \rr-order (when it exists) is denoted $a\wrr b$. 

(7) If, in $R$, the infimum of every pair of elements exists, then the ring is called \emph{\rr-good}. 

(8) A pair of elements $a,b\in R$ is called \rr-orthogonal if $a\wrr b=0$. It is easily checked that an orthogonal pair, i.e., $ab=0$, is \rr-orthogonal but not conversely.

(9) A ring $R$ is called \emph{weakly Baer} (wB) if for  $a\in R$, $\ann a$ is generated by an idempotent. Such rings are also called \emph{pp} or \emph{Rickart}. A ring $R$ is \emph{almost weakly Baer} (awB) if for $a\in R$, $\ann a$ is generated by idempotents. The name \emph{almost pp} is also used.  (See \cite{Be} and \cite{NR}.)

(10) The set of idempotents of $R$ (necessarily central) is denoted $\B(R)$. If $\B(R)= \{0,1\}$, then $R$ is called \emph{indecomposable}.

(11) For a (completely regular) topological space $X$, $\cx$ refers to the ring of continuous real valued functions on $X$ (see \cite{GJ} for all necessary details). For $f\in \cx$, the \emph{cozero set} of $f$ is $\{x\in X\mid f(x)\ne 0\}$, an open set, written $\coz f$; its complement is $\z(f)$. Notice that here, $h\rro f,g$ means that on $\cl(\coz h)$, $h, f$ and $g$ all coincide. If $X$ is connected, $\cx$ is indecomposable.

(12) In a topological space a \emph{clopen} set is one which is both open and closed.  

(13) The term ``regular ring'' means ``von Neumann regular''.

(14) The symbols $\R$, $\Q$ and $\N$ are for the the real numbers, the rational numbers and the natural numbers, respectively.\\[-.5ex]

The following shows some examples of \rr-good rings.
\begin{exs}\label{exsgood} The following types of reduced rings $R$ are \rr-good.

(1) All domains.

(2) All wB rings (including reduced regular rings).

(3) All rings of continuous functions $\cx$ where $X$ is a locally connected space. 

(4) All rings with three or fewer minimal primes. 

(5) Any ultrapower of \rr-good rings is \rr-good. 
 \end{exs}
 
 \begin{proof}
(1) Here the \rr-order is very coarse. For all $a\in R$, $0\rro a$ is the only relation and thus every pair of elements has an \rr-infimum. 
(2) In a wB ring the Pierce stalks of $R$ are domains and supports of elements are clopen in $\spec \B(R)$ (\cite[Lemma~3.1]{Be}). Then, \cite[Proposition~2.5]{BR3} applies. 
(3)~Here \cite[Theorem~3.5(1)]{BR3} gives the result.
(4)~This is \cite[Theorem~2.13(i)]{BR3}. 
(5)~Consider $P=\prod_{\alpha\in A}R_\alpha$ and a free ultrafilter $\mc{U}$ on $A$ and $I=\{(r_\alpha)\in P\mid \text{for some $U\in \mc{U}$}, r_\alpha =0\;\text{for all $\alpha\in U$}\}$. Let $S=P/I$. For $r\in P$, let $\bar{r} = r+I$. Clearly $S$ is reduced. The claim is that $S$ is \rr-good. Given $\bar{r},\bar{s}\in S$, for all $\alpha\in A$ let $r_\alpha \wrr s_\alpha = t_\alpha$. Then $t= (t_\alpha)_A \rro r,s$. The claim is that $\bar{t}= \bar{r}\wrr \bar{s}$.  Clearly $\bar{t}\rro \bar{r},\bar{s}$. Suppose $\bar{u}\rro \bar{r}, \bar{s}$. Then, $u^2-ur, u^2-us\in I$, and thus for some $U\in\mc{U}$ and all $\alpha\in U$, $u_\alpha \rro t_\alpha$ showing that $u^2-ut\in I$. Hence, $\bar{u}\rro \bar{t}$, as required.  (Note that Proposition~\ref{wrr}, below, could also have been used for a  proof.)
\end{proof}
 
 More information on the \rr-order is found in \cite{BR3} and will be quoted as needed.   \\[1ex]
 \setcounter{section}{2} \setcounter{thm}{0}
 \noindent\tbf{2. Some general results on the \rr-order.}  The first result, in spite of its simplicity, will be used later. Recall that a ring is \rr-good when it is a lower semi-lattice in the \rr-order. 
 \begin{lem} \label{bigidemp}  Let $R$ be a reduced ring which is \rr-good. For each $a\in R$, $\ann a$ contains a unique largest idempotent. \end{lem}

\begin{proof} Fix $a\in R$ and consider $e= 1\wrr (1-a)$. Then, $e^2=e1 =e$ and $ea=0$. If $f\in \B(R)$ and $fa=0$ then $f= f(1-a)$ and $f\rro 1, (1-a)$, showing $f\rro e$.  \end{proof}

Note that in Lemma~\ref{bigidemp}, the idempotent can be 0. It is shown in \cite[Theorem~2.6]{BR3} that an awB ring is wB if and only if it is \rr-good. This can be thought of as an illustration of the lemma. There are awB rings which are not wB, for example $\C(\beta\N \setminus \N)$ (\cite[Example~3.2]{NR}).

\begin{prop} \label{idempann}  (i)~Suppose that a reduced ring $R$ satisfies the condition that, for $a\in R$ if $\ann a \ne \bz$, there is a unique largest idempotent $0 \ne e \in \ann a$.  Then, $R$ is wB.  (ii)~If for each $a\in R$, $\ann a \ne \bz$ implies there is $0\ne e^2 =e \in \ann a$ and, moreover, $R$ is \rr-good, then $R$ is wB.  \end{prop}

\begin{proof} The last statement follows from the first by Lemma~\ref{bigidemp}.  Suppose $a\in R$ with $\ann a\ne \bz$ and let $0 \ne e \in \ann a= A$ be the unique largest idempotent in $A$.  If $eA = A$ then $A =eR$ and the conclusion holds. Otherwise, $(1-e)A \ne \bz$. Consider $a+e$ and note that $aR\oplus eR$ is a direct sum. Since $\ann (a+e)$ contains $(1-e)A$, it is non-zero.  Therefore if $0\ne f^2=f\in \ann (a+e)$, then $af = -ef $.  This shows that $af=eaf=0$ and $f\in A$, giving that $f\rro e$.  But then $f-ef = 0$, showing that $f=0$, a contradiction. Hence, $A= eR$, as required.   \end{proof}

The following two observations give information about \rr-infima and will be used in Section~3.

\begin{lem}\label{lbs} Let $R$ be a reduced ring and $r,s\in R$. If $0\ne k_1\rro r,s$ then $r\wrr (s+k_1) =0$ or there exists $0\ne k_2\rro r, s+k_1$ with $k_1k_2=0$ and $k_1\rro k_1+k_2\rro r,s$.  \end{lem}

\begin{proof}  Suppose that there is $0\ne k_2\rro r, s+k_1$. By definition, $k_1^2 = k_1r =k_1s$ and $k_2^2 = k_2r = k_2(s+k_1)$. From this, $k_1(r-s) =0$ and $k_2(r-s-k_1) =0$. Hence, $k_2(r-s) = k_2k_1$. Multiply this last equation by $r-s$ to get $k_2(r-s)^2 = k_2k_1(r-s) =0$, showing that $k_2(r-s)=0$. Then, $k_2(r-s-k_1)=0$ giving that $k_2k_1=0$.  
 
Now, $(k_1+k_2)^2= k_1^2+k_2^2= (k_1+k_2)r = k_1s+ k_2(s+k_1) =(k_1+k_2)s$.  Hence, $k_1+k_2 \rro r,s$  and also $k_1^2= k_1(k_1+k_2)$ showing $k_1\rro k_1+k_2\rro r,s$. \end{proof} 

\begin{cor} \label{basic} Let $R$ be a reduced ring and $r,s\in R$. If among all the non-zero \rr-lower bounds of $r$ and $s$ there is $k_1$ maximal in the \rr-order then $r\wrr (s+k_1)=0$. In particular if $k_1=r\wrr s$ then $r\wrr (s+k_1)=0$.  \end{cor}

\begin{proof} In the proof of Lemma~\ref{lbs} the expression  $k_1\le k_1+k_2\rro r,s$ shows that $k_2=0$. Hence, $r$ and $s+k_1$ have no non-zero \rr-lower bounds. \end{proof}

The following proposition, in the commutative case, connects the \rr-order and integral extensions. 

\begin{prop} \label{intcl}  Let $R$ be a commutative reduced ring which is integrally closed in some commutative \rr-good ring $S$.  Then, $R$ is \rr-good. \end{prop}

\begin{proof} Consider $r\in R$ and suppose $s\in S$ is such that $s\rro r$ in $S$. Then, $s$ satisfies the equation $x^2-rx=0$, showing that $s\in R$. Now consider $r_1, r_2\in R$ and let $s=r_1\wrr r_2$ in $S$. As above, $s\in R$ and $s\rro r_1, r_2$. If $t\in R$ is such that $t\rro r_1, r_2$ then this is true also in $S$ showing that $t\rro s$, which shows that $s=r_1 \wrr r_2$ in $R$. \end{proof}

\begin{cor} \label{clos}  Let $R$ be a commutative reduced ring embedded in a commutative \rr-good ring $S$. Then, the integral closure $T$ of $R$ in $S$ is \rr-good. In fact, the conclusion holds if $T$ is the quadratic integral closure of $R$ in $S$. \end{cor}

Corollary~\ref{clos} does not even need the entire quadratic integral closure, only that obtained by adjoining roots of polynomials of the form $x^2-rx$ from $S$. This is done by an induction, that is, by a union of a countable chain starting with $R= R_0$ and then $R_1= R_0[B_1]$, where $B_1=\{s\in S\mid s\;\text{is a root of}\; x^2-rx\;\text{for some $r\in R_0$}\}$, and so on. The closure to be used is $R_{\infty}=\bigcup_{n\ge 0}R_n$. For every pair $r_1, r_2\in R_\infty$, the element $r_1\wrr r_2$, as calculated in $S$, will be found in $R_\infty$. \\[-.5ex]

  \setcounter{section}{3} \setcounter{thm}{0}
 \noindent\tbf{3. On lifting \rr-orthogonal sets.}
  One of the topics of the paper \cite{BR3} was to investigate when, for a surjective ring homomorphism $\phi\colon R\to S$, a countable \rr-orthogonal set in $S$ could be lifted to an \rr-orthogonal set in $R$.  The first topic in this section is to look at a more restricted question: can \rr-orthogonal pairs be lifted via $\phi$?  Lemma~\ref{lbs} and Corollary~\ref{basic} will be used to get positive results but an example will show that this cannot always be done.
 
 The next simple lemma will be used in Proposition~\ref{noeth}, below. Recall that in $\spec R$, for $r\in R$, $V(r)= \{\mf{p}\in \spec R\mid r\in \mf{p}\}$, a basic closed set in the Zariski topology on $\spec R$, and $D(r)$ is its complement.

\begin{lem}\label{Zaris}  Let $R$ be a reduced ring. If $a,b\in R$ and $a\rro b$ then $V(b)\sbq V(a)$ and $D(a)\sbq D(b)$.
\end{lem}
 
\begin{proof} This follows from the equation $a^2=ab$. \end{proof}

  \begin{prop}\label{lift} Suppose $R$ and $S$ are reduced rings where $R$ is \rr-good and $\phi\colon R\to S$ is a surjective homomorphism with kernel $K$. If $s,s'\in S$ are such that $s\wrr s'=0$ and $r\in R$ with $\phi(r)=s$  then there is $r'\in R$ with $\phi(r') =s'$ and $r\wrr r' =0$.  \end{prop}
  
  \begin{proof}  Let $t\in R$ be such that $\phi(t) =s'$. If $r\wrr t=0$ there is nothing to prove. Otherwise, put $k=r\wrr t$ and note that $k\in K$. As in Corollary~\ref{basic}, $r\wrr (t+k)=0$ and if $r' =t+k$ then $\phi(r')=s'$. 
  \end{proof}
  
  Recall that in a noetherian space, ascending chains of open sets are finite. For a reduced ring $R$ it sometimes happens that $\spec R$, with the Zariski topology, is noetherian. This occurs if $R$ has the ACC on two-sided ideals but there are commutative non-noetherian examples.
  \begin{prop}\label{noeth}  Suppose $R$ and $S$ are reduced rings where $\spec R$ is noetherian and $\phi\colon R\to S$ is a surjective homomorphism with kernel $K$. If $s,s'\in S$ are such that $s\wrr s'=0$ and $r\in R$ with $\phi(r)=s$  then there is $r'\in R$ with $\phi(r') =s'$ and $r\wrr r' =0$.  \end{prop}
  
  \begin{proof} The notation of the proof of Proposition~\ref{lift} is used with $k_1$ instead of $k$. By Lemma~\ref{lbs}  a chain of lower bounds can be created $k_1\rro k_1+k_2 \rro \cdots $. According to Lemma~\ref{Zaris} and the noetherian condition, this process stops at an \rr-lower bound maximal among the \rr-lower bounds of $r$ and $s$, all of which are in $K$. Corollary~\ref{basic} then applies. \end{proof}
  
  There is a similar conclusion to that in Proposition~\ref{noeth} if it is assumed that $K$, as an $R$-module has left or right finite Goldie dimension. This is because it can be shown, using that $R$ is reduced, that when the sequence $k_1\rro k_1+k_2 \rro k_1+k_2+k_3 \rro \cdots$ is constructed, as in Lemma~\ref{lbs}, it can be seen that the sum $Rk_1+ \cdots + Rk_n$ is direct. Similarly on the right. Hence, the process stops at a maximal $k$ among the elements $\rro r,s$. 

It is not known at this time if this procedure can be modified to allow the lifting of an \rr-orthogonal set of more than two elements.  

The following is an example of a ring surjection $R\to S$ and an \rr-orthogonal pair $s_1, s_2\in S$ which does not lift to an \rr-orthogonal pair in $R$.  In fact, the stronger property that $s_1s_2=0$ will be imposed. Clearly $R$ cannot be \rr-good nor can $\spec R$ be noetherian by Propositions~\ref{lift} and \ref{noeth}. In the construction, the ring $K$ is taken to be a countable field but that is only for concreteness; any countable commutative domain could be used.
 
\begin{ex}\label{prod0} There exist reduced rings $R$ and $S$ with a surjective homomorphism $\phi\colon R\to S$ and a pair of non-zero elements $s_1, s_2$ in $S$ with $s_1s_2 =0$ such that for no pair $r_1, r_2\in R$ with $\phi(r_1)=s_1$ and $\phi(r_2)= s_2$ are $r_1$ and $r_2$ \rr-orthogonal. In particular, $r_1r_2\ne 0$.
\end{ex}

\begin{proof}  Throughout $K$ will be an arbitrary countable field. A polynomial ring $T= K[Y_1,Y_2, \{X_{ij}\}_{ij\in \N_0\times \N}]$, where $N_0=\N\cup\{0\}$, is used along with an ideal of relations that will define $R$. The purpose of the relations is to create non-zero lower bounds for pairs of non-zero elements. 

1. Set $\mc{L}_0=\{P\in K[Y_1,Y_2] \mid  Y_1Y_2\;\text{divides}\; P\}$. Then, $\mc{F}_0$ is the set of ordered pairs of distinct elements from $\mc{L}_0$, well-ordered in some way; $\mc{L}_0= \{\mbf{P}_{01}, \mbf{P}_{02}, \ldots \}$ with $\mbf{P}_{0n}=(P_{0n1}, P_{0n2})$.  The first relations are $$\mc{R}(0n1)= X_{0n}^2- X_{0n}(Y_1+P_{0n1})$$ and, $$\mc{R}(0n2)= X_{0n}^2- X_{0n}(Y_2+ P_{0n2})\;,$$ for each element of $\mc{F}_0$. 

2. Set \begin{equation*}\begin{array}{ccc}\mc{L}_1&=&\{P\in K[Y_1,Y_2, \{X_{0n}\}_{n\in \N} \mid P\;\text{has 0 constant term, and any} \\&&\text{term without a factor}\; Y_1Y_2 \;\text{has a factor some $X_{0n}$}\}\end{array}\end{equation*} Then, $\mc{F}_1$ is the set of ordered pairs $(F,G)$ of distinct elements  from $\mc{L}_1$ such that at least one of $F$ and $G$ is not in $\mc{L}_0$. These are well-ordered in some way $\mc{F}_1= \{\mbf{P}_{11}, \mbf{P}_{12}, \ldots\}$ so that $\mbf{P}_{1n} =(P_{1n1}, P_{1n2})$. The relations are $$\mc{R}(1n1)= X_{1n}^2- X_{1n}(Y_1+P_{1n1})$$ and $$ \mc{R}(1n2)= X_{1n}^2- X_{1n}(Y_2+P_{1n2})\,,$$ for each element of $\mc{F}_1$. 

3. For $n>0$ assume that the sets $\mc{L}_0, \ldots, \mc{L}_{n-1}$, $\mc{L}_0\subsetneq \mc{L}_1 \subsetneq \cdots \subsetneq \mc{L}_{n-1}$,  and also the sets of pairs $\mc{F}_0,\ldots, \mc{F}_{n-1}$ have been constructed. Put 
\begin{equation*}\begin{array}{ccc}\mc{L}_n&=& \{P\in K[Y_1,Y_2, \{X_{0j}\}_{j\in\N}, \ldots, \{X_{n-1,j}\}_{j\in\N}] \mid P\;\text{has 0 constant}\\&& \text{term and any term without a factor $Y_1Y_2$ has some }\\&&\text{ $X_{ij}$ as a factor with}\;i=0,\ldots, n-1\}\,.\end{array} \end{equation*} 
Let $\mc{F}_n$ be the set of ordered pairs $(F,G)$ of distinct elements from $\mc{L}_n$ such that at least one of $F$ and $G$ is not in $\mc{L}_{n-1}$. These are well-ordered in some way $\mc{F}_n= \{\mbf{P}_{n1},\mbf{P}_{n2}, \ldots \}$ where $\mbf{P}_{nj}= (P_{nj1}, P_{nj2})$. The relations are $$\mc{R}(nj1)= X_{nj}^2- X_{nj}(Y_1+P_{nj1})$$ and $$\mc{R}(nj2)= X_{nj}^2- X_{nj}(Y_2+P_{nj2})\;,$$ for each element of $\mc{F}_n$. 

The ideal $I$ of $T$ is generated by $\bigcup_{n\in \N_0, j\in \N}(\mc{R}(nj1) \cup \mc{R}(nj2))$ and $\sqrt{I}$ is its radical. The ring $R= T/\sqrt{I}$, where for $P\in T$, $P+\sqrt{I}$ is written $p$. The ring $S$ is $R/J$ where $J= (y_1y_2)\cup (\bigcup_{ij\in \N_0\times \N} x_{ij})$, where for $p\in R$, $p+J$ is denoted $\bar{p}$. The ring $S$ is isomorphic to $K[Y_1,Y_2]/(Y_1Y_2)$, as will be seen.

It will first be shown that, for $ij \ne kl$, $X_{ij} -X_{kl}\notin \radi$.  If this is not the case, there is $m\in \N$ with $L=(X_{ij}-X_{kl})^m\in I$. Then, 
$$(*)\;\; L= \sum_{abc\in \N_0\times \N\times \{1,2\}} \mc{R}(abc)f_{abc}, \;\;f_{abc}\in T\;. $$

The idea is to isolate those terms of ($*$) which have a factor which is a variable other than $X_{ij}$ or $X_{kl}$. For each $f_{abc}$ write $f_{abc}'$ to be the sum of the terms of $f_{abc}$ which are not in $K[X_{ij},X_{kl}]$ and then $f_{abc}= f_{abc}'+ f_{abc}''$. Given the form of $L$ it follows that 
\begin{equation*}\begin{split}
(**)\;\; \left(\sum_{abc,\,ab\ne ij, kl} \mc{R}(abc)f_{abc}\right) + \mc{R}(ij1)f_{ij1}' +\mc{R}(ij2)f_{ij2}'\\ +\mc{R}(kl1)f_{kl1}' + \mc{R}(kl2)f_{kl2}'\\ -X_{ij}(Y_1+P_{ij1})f_{ij1} - X_{ij}(Y_2+ P_{ij2})f_{ij2}\\- X_{kl}(Y_1+P_{kl1})f_{kl1} - X_{kl}(Y_2+ P_{kl2})f_{kl2} \\ +X_{ij}^2(f_{ij1}'+f_{ij2}') + X_{kl}^2(f_{kl1}'+ f_{kl2}') =0
\end{split} \end{equation*} These are precisely the terms of ($*$) which have factors other than $X_{ij}$ and $X_{kl}$. It should be kept in mind that $P_{ijc}$ has no terms involving $X_{uv}$, where $u\ge i$, and similarly for $kl$. This means that 
$$L= X_{ij}^2f_{ij1}'' + X_{ij}^2f_{ij2}'' + X_{kl}^2f_{kl1}'' + X_{kl}^2f_{kl2}''\;. $$
Notice that $m=1$ cannot happen since the only appearance of $X_{ij}$ or $X_{kl}$ without other factors is as a square in $\mc{R}(ijc)$. It follows that, $m\ge 2$.
Hence, $f_{ij1}''$ has a term $kX_{ij}^{m-2}$ and $f_{ij2}''$ has a term $(1-k)X_{ij}^{m-2}$, for some $k\in K$. At least one of these is non-zero. This means that in ($*$) there are terms $kX_{ij}^{m-1}Y_1$ and $(1-k)X_{ij}^{m-1}Y_2$.  These terms occur nowhere else in ($*$), contradicting the equation ($**$). It is not necessary to look at $X_{kl}$. 

Thus, for $ij\ne kl$, $X_{ij}$ and $X_{kl}$ are distinct module $\radi$.  In a similar fashion it can be seen that no $X_{ij}$ is in $\radi$. 

Next, can an element  $P\in K[Y_1,Y_2]$ be in $\radi$?  The form of the relations shows that no power of $P$ can be in $I$. This shows that $R/(\{x_{ij}\mid ij\in \N_0\times \N\})$ is indeed the polynomial ring $K[Y_1,Y_2]$ and that $R/J\cong K[Y_2,Y_2]/(Y_1Y_2)$.

A lifting of the pair $(\ov{y_1}, \ov{y_2})$ to $R$ has the form $(y_1+ p_1, y_2+p_2)$, where $p_1$ and $p_2$ are images of polynomials $P_1$ and $P_2$ in $T$, with 0 constant term, and it may be assumed that $P_1\ne P_2$. The pair $(P_1,P_2)\in \mc{L}_n$, for some $n\ge 0$. As planned, the relations $\mc{R}(ijk)$ were chosen so that each lifting of the pair $(\ov{y_1}, \ov{y_2})$ to $R^2$ has a non-zero lower bound in $R$.  More exactly, if $(P_1,P_2) = (P_{nj1}, P_{nj2})$, then $x_{nj}$ is a non-zero lower bound of $p_1$ and $p_2$. Note that this also means that no such lifted pair has product 0.
\end{proof}

The next observation is about lifting a countable \rr-orthogonal set.

\begin{prop}\label{idemp}  Let $R$ be an \rr-good ring and $\phi\colon R\to S$ a surjective ring homomorphism which preserves \rr-infima of pairs of elements. Then, a countable set of pairwise \rr-orthogonal elements of $S$ can be lifted to a set of pairwise \rr-orthogonal elements of $R$. \end{prop}

\begin{proof} Suppose $\{s_1,s_2, \ldots \}$ is a set of pairwise \rr-orthogonal elements of $S$. Lift $s_1$ to $r_1\in R$. Assume, for $n\ge 1$ that $\{s_1, \ldots, s_n\}$ have been lifted to pairwise \rr-orthogonal elements $\{r_1, \ldots ,r_n\}$ in $R$, where $\phi(r_i) =s_i$. For $i=1,\ldots, n$, pick some $t_i\in R$ such that $\phi(t_i)= s_{n+1}$ and $t_i\wrr r_i=0$, which is possible by Proposition~\ref{lift}. Then set $r_{n+1} = (\cdots(t_1\wrr t_2)\wrr t_3) \cdots) \wrr t_n$.  By the assumption on $\phi$, $\phi(r_{n+1}) =   (\cdots(s_{n+1}\wrr s_{n+1})\wrr s_{n+1}) \cdots \wrr s_{n+1} = s_{n+1}$. By construction, for $i=1,\ldots, n$, $r_{n+1}\wrr r_i=0$ since $r_{n+1}\wrr r_i\rro t_i\wrr r_i=0$. 
\end{proof}

There are cases where Proposition~\ref{idemp} applies. For one such,  \cite[Proposition~2.1(iv)]{BR3} can be improved. It is not necessary to assume that $S$ is \rr-good as was done earlier.

\begin{prop} \label{wrr}  Let $R$ be a reduced ring and $I$ an ideal generated by idempotents.  Put $S=R/I$ and let $\phi\colon R\to S$ be the natural surjection. If $a\wrr b =c$ exists in $R$ then $\phi (c) = \phi(a)\wrr \phi(b)$. Moreover, if $R$ is \rr-good then $S$ will be \rr-good and, in addition, countable \rr-orthogonal sets in $S$ lift to \rr-orthogonal sets in $R$. 
\end{prop}

\begin{proof}  Let $a,b\in R$ with $c=a\wrr b$. The aim is to show that $\phi(c)= \phi(a)\wrr \phi(b)$. Certainly $\phi(c) \rro \phi(a), \phi(b)$.   Suppose $x\rro \phi(a),\phi(b)$ and pick $y\in R$ with $\phi(y)=x$. Then, there are $r, r'\in R$ and $e,f\in I\cap \B(R)$ with $y^2-ya=er$ and $y^2-yb=fr'$.  Put $g= (1-e)(1-f)$. Then, direct calculation shows that $y^2g-yga = y^2g- ygb =0$ showing that $yg\rro a,b$ and, thus, $yg\le c$. However, $\phi(g)=1$ and $\phi(y)=x\rro \phi(c)$. Hence, $\phi(c)=\phi(a)\wrr \phi(b)$.  

The last statement follows since any pair in $S$ will lift to a pair in $R$ which has an \rr-infimum which is preserved by $\phi$ and Proposition~\ref{idemp} applies. \end{proof}

  \setcounter{section}{4} \setcounter{thm}{0}
 \noindent\tbf{4. The \rr-order and localizations.} This section looks at mostly commutative rings and the connections between the \rr-order and localizations. The results are mixed in that under some conditions the localizations behave well with respect to the \rr-order while under others the opposite is true. 

\begin{prop}\label{Sinv}  Let $R$ be an \rr-good ring and $S$ a multiplicatively closed set of central non-zero divisors in $R$. 

(i) The passage from $R$ to $RS\inv$ preserves \rr-infima if and only if for all $a,b\in R$ and $s\in S$, $as\wrr bs = (a\wrr b)s$.

(ii) The passage from $R$ to $RS\inv$ does not, in general, preserve \rr-infima in $R$.

(iii) If $R$ is wB then \rr-infima are preserved going from $R$ to $RS\inv$.\end{prop}

\begin{proof} (i) Suppose first that if $r = a\wrr b$ in $R$, for all $t\in S$, $rt= at\wrr bt$.  If $ds\inv \rro a,b$ in $RS\inv$ then $d\rro as, bs$ in $R$.  Hence, $d\rro rs$ and $ds\inv \le r$  in $RS\inv$ by \cite[Proposition~2.1(ii)]{BR3}, which says that central units distribute over \rr-infima.   

Suppose now that \rr-infima are preserved and that $r=a\wrr b$ in $R$ and in $RS\inv$.  By \cite[Proposition~2.1(ii)]{BR3} again, $rs\inv = as\inv \wrr bs\inv$ and $rs= as\wrr bs$ in $RS\inv$ for all $s\in S$, and this will be true in the subring $R$.

(ii) There is an example using $R = \C(\R)$ and $S=\{h^n \mid n\in \N\}$ where $h\in \C(\R)$ is such that $\z(h) = \{0,1\}$. Note that $h$ is a non-zero divisor. Let $f, g\in \C(\R)$ such that $f$ and $g$ are different everywhere except on the interval $[0,1]$, where they coincide and are non-zero. By construction $f\wrr g=0$ (see Item (11) in Definitions and Terminology). However, $fh \wrr gh \ne 0$. In fact, $fh\wrr gh=k$ where $k$ coincides with $fh$ and $gh$ on the interval $(0,1)$ and $\coz k=(0,1)$. Then, $0\ne kh\inv \rro f,g$ in $RS\inv$. The conclusion follows by (i).  Note that, here, $R$ is \rr-good.

(iii) By the proof of \cite[Theorem~2.6]{BR3}, for $a,b\in R$, $a\wrr b =ea=eb$, where $e=e^2$ and $eR=\ann(a-b)$. For $s\in S$, $\ann (as-bs) =\ann (a-b) =eR$, since $s$ is a non-zero divisor. Hence, $as\wrr bs= eas=ebs$ showing that $as\wrr bs= (a\wrr b)s$.    \end{proof}

\begin{cor} \label{intS}  Suppose $R$ is a commutative reduced ring and $S\sbq R$ a multiplicatively closed set of non-zero divisors. If $R$ is \rr-good and integrally closed in $RS\inv$ then $RS\inv$ is  \rr-good. \end{cor}

\begin{proof} First note that Proposition~\ref{Sinv}~(i) applies since, for $a,b\in R$,  if $cu\inv \rro a,b$ in $RS\inv$, for some $u\in S$, then $cu\inv$ satisfies $x^2-xa=0$ (and $x^2-xb=0$), showing that $cu\inv \in R$. Now consider $as\inv, bt\inv \in RS\inv$. Put $d=at\wrr bs$ in $R$; this is also the \rr-infimum in $RS\inv$.  But then, $ds\inv t\inv = as\inv \wrr bt\inv$ since $st$ is now a unit.    \end{proof}

 \begin{ex}\label{nogood}  There is  a commutative, reduced ring $R$ which is not \rr-good such that $\qcl(R)$ is \rr-good. \end{ex}
 
 \begin{proof} Fix a prime integer $p$ and let $R$ be the ring of sequences from $\Z$ which are eventually constant modulo $p$. This ring is not \rr-good (\cite[Example~2.12]{BR3}). However, $\qcl(R) = \prod_{n\in \N}\Q = S$. This is because $R\sbq \prod_{n\in \N}\Z \sbq S$ showing that $\qcl(R)\sbq S$. Moreover, $(p,p,\ldots)\in R$ showing that $(1/p,1/p, \ldots) \in \qcl(R)$. For any sequence of integers $(n_1, n_2, \ldots)$, $(n_1p, n_2p, \ldots) \in R$ giving $(n_1,n_2, \ldots )\in \qcl(R)$. This gives $\prod_{n\in \N}\Z\sbq \qcl{R}$ and, hence, all of $S$ in $\qcl(R)$.  \end{proof}
 
 The ring $\cx$ of \cite[Theorem~3.5(2)]{BR3} is another example since $\qcl(\cx)$ is regular.
 
 Notice also that if $A$ is the \rr-good ring of sequences of integers which are eventually constant, then $Q(A)$, the complete ring of quotients is regular, in fact $Q(A) =\prod_{n\in \N}\Q$. However, there is the non-\rr-good ring $R$ of Example~\ref{nogood} lying between $A$ and $Q(A)$. 

In the example found in the proof of Proposition~\ref{Sinv}~(ii) it can be noted that $\C(\R)$ is \rr-good as is its classical ring of quotients, $\qcl(\R)$, which is even von Neumann regular. However, the passage from $\C(\R)$ to $\qcl(\R)$ does not preserve \rr-infima.  

The next remark shows that localization at a prime need not preserve \rr-infima.

\begin{remark}\label{RsupP} Let $R$ be a commutative reduced ring and $\mf{p}$ a prime ideal of $R$. Then the passage from $R$ to $R_{\mf{p}}$ need not send \rr-infima to \rr-infima.\end{remark}

\begin{proof} Again take $R = \C(\R)$ and $\mf{m}=\{ f \in \C(\R)\mid f(0)=0\}$, a maximal ideal. Put $T = \C(\R)_{\mf{m}}$ and let $\psi\colon \C(\R) \to T$ be the localization map with kernel $K= \{g\in \C(\R) \mid \exists f\notin \mf{m}\;\text{with}\; fg=0\}$. For $h\in \C(\R)$ denote $\psi(h) = \bar{h}$. Consider $h, k\in \C(\R)$ such that $h\wrr k=0$ but such that there is a neighbourhood $N$ of 0 where $h$ and  $k$ coincide and are non-zero. Then, $h, k\notin \mf{m}$ and $\bar{h}$ and $\bar{k}$ are invertible in $T$. Moreover, $h-k\in K$ so that $\bar{h}=\bar{k}$ in $T$. Hence, $\bar{h}\wrr \bar{k} = \bar{h} = \bar{k} \ne 0$ while $h\wrr k=0$.   (As examples, the following functions may be used for $h$ and $k$. Let $h$ be $-x$ on $(-\infty, -1)$, $x$ on $(1, \infty)$ and constantly 1 on $[-1,1]$. Let $k$ be $-2x+1$ on $(-\infty, -1)$, $2x -1$ on $(1, \infty)$ and constantly 1 on $[-1,1]$.)\end{proof}

The following lemma will be used in the example below. 

\begin{lem} \label{Bvide}  Let $R$ be an indecomposable reduced ring. Suppose $u$ is a central unit and $b\in R$ arbitrary. If $c\rro u,b$, then $c=0$ or $c=u=b$. In any case, $u\wrr b$ exists. \end{lem}

\begin{proof} Suppose $c\rro u,b$ which implies that $c^2-cu=c^2- cb=0$. From this, $c^2u^{-2} =cu\inv$, an idempotent. Thus, $cu\inv =0$ or $cu\inv =1$, or, $c=u$ or $c=0$. If $c=u$ then $u^2-ub=0$ implies that $c=u=b$, as well. The conclusion follows.
 \end{proof}

\begin{ex} \label{Rbarnot}  There is an example of a commutative reduced \rr-good ring $R$ and a multiplicatively closed set $S\sbq R$ not containing 0 such that the image of $R$ is not \rr-good and, moreover, $RS\inv$ is also not \rr-good. Hence, a  reduced commutative local ring is not necessarily \rr-good.\end{ex}

\begin{proof} Let $R = \C(\R^2)$, an indecomposable \rr-good ring.  Consider the fixed maximal ideal $\mf{m} =\mf{m}_\bz= \{f\in R\mid f(\bz) =0\}$, where $\bz =(0,0)$.  (Any $\mf{m}_p$ would work equally well.) Then the kernel of the localization at $\mf{m}$ is $K=\{g\in R\mid \z (g) \;\text{is a neighbourhood}\text{ of $\bz$}\}$. (This ideal is called $\mc{O}_\bz$ in \cite[4I]{GJ}). For $f\in R$, $f+K$ will be written $\bar{f}$. 

It is first necessary to say what it means for $\bar{h}\rro \bar{f}, \bar{g}$. This is equivalent to saying that $h^2-hf$ and $h^2-hg$ are in $K$, or, using the fact that $S = R\setminus \mf{m}$ is multiplicatively closed, that there is $l\notin \mf{m}$ such that $(h^2-hf)l = (h^2-hg)l = 0$.  

Lemma~\ref{Bvide} can be used to show that any example of $f,g\in R$ where $\bar{f}$ and $\bar{g}$ have no \rr-infimum in $\bar{R}$ or $RS\inv$ will have to be elements of $\mf{m}$.

Two functions will now be defined, by descriptions of their graphs, so that their images in $R/K$ and in $R_\mf{m}$  do not have an \rr-infimum. For $n\in \N$ consider rays $\rho_n$ in the first quadrant of $\R^2$ starting at $\mbf{0}$ having slope $n$ and the positive vertical axis $\rho_\infty$. The open region between $\rho_n$ and $\rho_{n+1}$ will be denoted $\mc{S}_n$. The functions $f$ and $g$ will be zero on the rays $\rho_n$ and $\rho_\infty$, in the second, third and fourth quadrants and in the region in the first below $\rho_1$. 

For each $n\in \N$, consider the closed disk $B_n$ of radius $1/n$ with centre $\mbf{0}$.  In $B_n\cap \mc{S}_n$, $f$ and $g$ will coincide. For $0<r\le 1/n$ consider the points $p$ and $q$ on $\rho_n$ and $\rho_{n+1}$, respectively, with $\Vert p\Vert = \Vert q \Vert = r$ and the line segment joining $p$ and $q$. The graphs of both $f$ and $g$ over this line segment are semicircles of radius $r$.  Let $p\in \mc{S}_n$ with $\Vert p\Vert =1/n$ and $q$ on the ray  from $\mbf{0}$ through $p$ with $\Vert q\Vert > 1/n$. Then, $f(q) = f(p)$ while $g(q) = g(p)(1/(\Vert q\Vert n)) = f(p)(1/(\Vert q\Vert n))$. Both $f$ and $g$ on $\mc{S}_n$ are bounded by $1/n$, necessary for continuity on the ray $\rho_\infty$. They coincide at $p$ if $\Vert p\Vert \le 1/n$ but they differ if $\Vert p\Vert > 1/n$.  

 For $n\in \N$, let $h_n\in R$ be the function which coincides with $f$ on $\mc{S}_n$ and is zero elsewhere. Then, $h_n\in R$ since $f$ is 0 on the boundary of $\mc{S}_n$. For $r>0$, let $D_r$ be the open disk around $\mbf{0}$ with radius $r$. Recall that each open set in $\R^2$ is the cozero set of some continuous function. If $k\in R$ has cozero set $D_{1/n}$ then $(h_n^2-h_nf)k = (h_n^2-h_ng)k =0$. This shows that $\ov{h_n}\rro \bar{f}, \bar{g}$ and that $\ov{h_n}\ne \bar{0}$, since $\bar{k}$ is a non-zero divisor in $R/K$. Hence, $\bar{f}$ and $\bar{g}$ have non-zero \rr-lower bounds in $R/K$.  

Now suppose that $\bar{h} = \bar{f} \wrr \bar{g}$, for some $h\in R$. For each $n\in \N$, $\ov{h_n}\rro \bar{h}$. There is $r>0$  and $l\in R$ with cozero set $D_r$ such that $(h^2-hf)l = (h^2-hg)l =0$. Multiplying these equations by $l$ shows that $hl\rro fl, gl$, in $R$. This means that for $p\in \R^2$, $\Vert p\Vert<r$, $h(p) =0$ or $h(p) = f(p)=g(p)$. Pick $n\in \N$ with $1/n<r$. There exists $0<r'<1/n$ and $u\in R$ with $\coz u=D_{r'}$ such that $h_nu\rro fu, gu$ and $h_nu\rro hu$. Hence, for $q\in \mc{S}_n$, $0< \Vert q\Vert <r'$, $h_n(q) = h(q)= f(q) =g(q)$. With $p\in \mc{S}_n$ and $1/n <\Vert p\Vert <r$, $hl(p) = 0$ or $hl(p) = fl(p)= gl(p)$. However, here $fl(p)\ne gl(p)$ showing that $hl(p) =0$ and $h(p)=0$. 

Let $L$ be the line segment (in $\mc{S}_n$) connecting $p$ and $q$. Then $h\res_L(q) =f\res_L(q)\ne 0$ and $h\res_L(p) =0$ and $\coz h_L$ is open in $L$. Let $a$ be in the boundary of $\coz h_L$ in $L$ and $\{b_m\}$ a sequence in $\coz h_L$ converging to $a$. Each $h(b_m) = f(b_m)$ while $\{h(b_m)\} \to 0$ and $\{f(b_m)\}\to f(a) \ne 0$. This is a contradiction showing that $\bar{f}$ and $\bar{g}$ do not have an \rr-infimum in $R/K$.  

The next step is to show that $\bar{f}$ and $\bar{g}$ have no \rr-infimum in $R_\mf{m}$. Suppose that $\bar{a}\bar{b}\inv = \bar{f}\wrr \bar{g}$ in $R_\mf{m}$, $a\in R$ and $b\in R\setminus \mf{m}$. By \cite[Proposition~2.1(ii)]{BR3}, one more time,  $\bar{a} = \bar{f}\bar{b} \wrr \bar{g}\bar{b}$ in $R_\mf{m}$ and, also, in the subring $R/K$. For some $r>0$ and $l\in R$ with $\coz l=D_r$, $al \rro fbl, gbl$. By reducing $r$, if necessary, it may be assumed that $D_r\sbq \coz b$.  For all $n\in \N$ with $1/n<r$, it is clear that $\ov{h_n}\bar{b}\rro \bar{a}$.  The argument used in $R/K$ can be applied to these $h_nb$ and $a$ to again arrive at a contradiction. That is, for $1/n<r$, and $p\in \mc{S}_n$ with $\Vert p\Vert <1/n$, $h_nb(p) = fb(p) =gb(p) =a(p)\ne 0$, while for $q\in \mc{S}_n$, $1/n <\Vert q\Vert <r$, $a(q)=0$ and $fb(q)\ne 0$. Hence, $R_\mf{m}$, a commutative, reduced local ring,  is not \rr-good. \end{proof}

Notice, by the way, in the situation of the proof of Example~\ref{Rbarnot}, that in $R$, $f\wrr g=0$.   \\[-.5ex]

  \setcounter{section}{5} \setcounter{thm}{0}
\noindent\tbf{5. The \rr-order and the Pierce sheaf.}
This section looks at a characterization of rr-good rings in terms of the Pierce sheaf. There are interesting and suggestive necessary conditions but they turn out not to be sufficient. The complete set of conditions is hard to test for.

Information about the Pierce sheaf can be found in \cite[\tbf{V}~2]{J}. Here are a few basic elements. The sheaf for a ring $R$ has base space $\mc{X}= \spec\B(R)$, a boolean space (compact, Hausdorff and totally disconnected). The Pierce stalks are defined as follows: for $x\in \mc{X}$, $R_x= R/xR$. For $r\in R$, $r+Rx$ is written $r_x$. In the case where $R$ is reduced, so are all the Pierce stalks and $\B(R_x)$ has no non-trivial idempotents, i.e, the Pierce stalks are indecomposable. The topology on the \emph{espace \'etal\'e} is generated by the sets: for $r\in R$ and $Y$ a clopen set in $\mc{X}$, $\{r_y\mid y\in Y\}$. The ring of sections of the sheaf is isomorphic to $R$. The support in $\mc{X}$ of an element of $R$ is closed in $\mc{X}$. The key fact to be used is that if $r,s\in R$ coincide at $x\in \mc{X}$, they also coincide on a clopen set containing $x$. The sections of the sheaf are identified with the elements of $R$. For $r\in R$, $\supp r = \{x\in \mc{X}\mid r_x\ne 0_x\}$ and its complement is written $\z(r)$. 

Recall (\cite[Proposition~2.5]{BR3}) that a ring $R$ is awB if and only if its Pierce stalks are domains.

The following two results can be considered as corollaries of Lemma~\ref{Bvide} which dealt with indecomposable rings.

  \begin{cor} \label{unitwedge}  Let $R$ be a reduced ring and $u, b\in R$ where  $u$ is a central unit. Then, $u\wrr b$ exists if and only if $U=\{x\in \spec\B(R) \mid u-b \in Rx\}$ is clopen in $\spec\B(R)$. \end{cor}
 
 \begin{proof}
 For the computation of $u\wrr b$ it suffices to consider $1$ and $b$, by replacing $b$ by $u\inv b$. (\cite[Proposition~2.1(ii)]{BR3} is used again.) Suppose $c=1\wrr b$ exists. Then, since $c\le 1$, $c\in \B(R)$. For $x\in \spec \B(R)$, it follows that $c_x=0_x$ or $c_x=1_x$. In the latter case, $c_x=1_x=b_x$. Hence, $\supp c\sbq U$.  Since $U$ is open and $\supp c$ is closed (as well as open), $V=U\setminus \supp c$ is open. If $V\ne \vide$   then there is $0\ne e\in \B(R)$ with $\supp e \sbq V$ and, also, $e^2 =e=eb$. It follows that $e\rro 1, b$ and $e\rro c$, a contradiction unless $V=\vide$. Hence, $U=\supp c$ and is clopen.
  
 In the other direction, if $U$ is clopen with $c$ the idempotent with support $U$ then $c =cb$.  If $e\in \B(R)$ is such that $e\rro 1,b$ then the support of $e$ is contained in $U$ making $e\rro c$. 
 
In the case that $U=\vide$ then $1\wrr b=0$. \end{proof}

Notice that in Corollary~\ref{unitwedge} the set $U$ is, in fact, $\z(u-b)$. The next corollary recalls \cite[Theorem~2.6]{BR3} but, here, the Pierce stalks do not need to be domains. 
 
  \begin{cor} \label{PiercNec}  Let $R$ be a reduced ring which is \rr-good. Then, the support in the Pierce sheaf of each $a\in R$ is clopen and each Pierce stalk $R_x$ is \rr-good. 
 \end{cor}
 
 \begin{proof} Corollary~\ref{unitwedge} will be used. For each $a\in R$, consider $c = 1\wrr (1-a)$. Then $c\in \B(R)$ and $\supp (1-a)$ is a clopen set since, by Corollary~\ref{unitwedge}, $\z(1-a)$ is clopen. If, for $x\in \mc{X}$, $c_x= 1_x$ then $c=c(1-a)$ which says that $a_x=0_x$.  On the other hand, if $a_x=0_x$ then there is a clopen neighbourhood $N$ of $x$ such that for all $y\in N$, $a_y=0_y$. Let $e_N\in \B(R)$ be such that $\supp e_N =N$. Then, $e_N =e_N(1-a)$ and, hence, $e_N\rro c$, which shows that $c_x=1_x$. In other words,  $\z(a) = \supp c$ is a clopen set in $\mc{X}$, showing that $\supp a$ is clopen.
 
 The Pierce stalks are \rr-good by Proposition~\ref{wrr}.
 \end{proof}
 
The conditions of Corollary~\ref{PiercNec} are not sufficient as will be shown by the next lemma and an example. A characterization of \rr-good rings will follow.

\begin{lem} \label{seqinf}  Let $A$ be an \rr-good indecomposable ring and $B$ an \rr-good subring. Suppose that there are $b_1,b_2\in B$ such that $b_1\wrr b_2 =c$ when calculated in $B$ and $b_1\wrr b_2 =d$ when calculated in $A$. Assume further $c\ne d$. Then, when $R$ is the ring of sequences from $A$ which are eventually constant in $B$,  the ring $R$ satisfies the conditions of Corollary~\ref{PiercNec} but is not \rr-good.
\end{lem}

\begin{proof} Note that $c\rro d$ and $d\notin B$. Let $r\in R$ be the sequence which is constantly $b_1$ and $s\in R$ that which is constantly $b_2$. The topology on $\spec R$ is the one-point compactification of $\N$, $\N \cup\{\infty\}$. The Pierce stalks of $R$ can be denoted $\{R_n\}_{n\in \N}$, each $R_n=A$ and $R_\infty=B$. Notice that the Pierce support of any element is clopen. The aim is to show that $r\wrr s$ does not exist.  Suppose $u=r\wrr s$. By Proposition~\ref{wrr}, for all $n\in \N$, $u_n= r_n\wrr s_n= b_1\wrr b_2= d$. This shows that $u$ does not have the form of an element of $R$. Hence $r$ and $s$ do not have an \rr-infimum. 
\end{proof} 

Note that the proof of Lemma~\ref{seqinf} would also apply if $R$ were the ring of sequences from $A$ eventually in $B$ (but not necessarily constant). The conclusion is the same and the proof uses the same elements as in the lemma but, in this case, $\B(R)$ is a complete boolean algebra. 

It should be recalled how, in a ring of the form $\cx$, two elements $f,g\in\cx$ can have a non-zero lower bound. In fact $0\ne h\rro f,g$ exactly when on $\coz h$, $f,g$ and $h$ coincide and, therefore, for $x$ on the boundary of $\coz h$, $f(x) = g(x) =0$. See \cite[Section~3]{BR3} for many more details.

\begin{ex} \label{notsuff} Let $X$ be the subspace of $\R^2$, $X=\{(x,y)\mid 0\le x\le 2\}$ and $Y$ the dense open subspace obtained by deleting the half-line $\{(1,y)\mid y\ge 0\}$. Then, $\C(X)\sbq \C(Y)$, $\C(X)$ and $\C(Y)$ are \rr-good and indecomposable but there are elements $f,g\in \C(X)$ which satisfy the conditions of Lemma~\ref{seqinf}.  \end{ex}

\begin{proof} Since both $X$ and $Y$ are connected, the rings are indecomposable. Further, since $X$ and $Y$ are locally connected, the rings are \rr-good. (See \cite[Theorem~3.5(1) and  Proposition~3.10]{BR3}.) Notice that $\cx\sbq \C(Y)$ because $Y$ is dense in $X$ so that the restriction $\cx\to \C(Y)$ is one-to-one. 
The functions $f$ and $g$ in $\cx$ will be defined.  

(i) For $y\le 0$, $f(x,y)= g(x,y) =0$. 

(ii) For $y_0>0$ and $0\le x\le 1/2$, $f(x,y_0)$ is linear between $f(0,y_0) = y_0$ and $f(1/2, y_0) = 0$.

(iii) For $y_0>0$ and $1/2 \le x\le 1$, $f(x,y_0)$ is linear between $f(1/2, y_0)=0$ and $f(1, y_0)=y_0$.

(iv) For $y_0> 0$ and $1\le x\le 2$ the same pattern is repeated between $1\le x\le 3/2$ as in (ii) and similarly for $3/2\le x\le 2$ as in (iii).

(v) For $y_0>0$ and $0\le x\le 1/2$ and $1\le x\le 3/2$, $g(x,y_0)$ coincides with $f(x,y_0)$. 

(vi) For $y_0>0$ and $1/2\le x\le 1$, $g(x, y_0)$ can be taken to be piecewise linear going from $g(1/2,y_0) = 0$ to $g(3/4, y_0)= y_0$ and then horizontally to $g(1, y_0)= y_0$.

(vii) For $y>0$ and $3/2\le x\le 2$, $g$ behaves as in (vi). 

Notice that $f$ and $g$ are nowhere equal in $\{(x,y) \mid y>0, 1/2< x <1\}$ and in $\{(x,y)\mid y>0, 3/2<x< 2\}$.

The first thing to notice is that in $\C(X)$, $f\wrr g= h$, where $h$ coincides with $f$ and $g$ in $\{(x,y)\mid y>0, 0\le x\le 1/2\}$ and is zero elsewhere. However, in $Y$, $f\res_Y\wrr g\res_Y$ is $h\res_Y+k$, where $k$ coincides with $f$ and $g$ in $\{(x,y)\mid y>0, 1<x\le 3/2\}$ and is zero elsewhere. Clearly $h\res_Y\ne h\res_Y+k$ and $h\res_Y\rro h\res_Y+k$. The point is that by removing the half-line the continuity along that half-line is no longer at issue which allows the larger \rr-lower bound. 

 The hypotheses of Lemma~\ref{seqinf} are satisfied.
 \end{proof}
 
 The following characterization of \rr-good rings in terms of their Pierce sheaf does not readily translate into properties of the ring structure but is worth noting nonetheless. The property about supports in Corollary~\ref{PiercNec} is not needed in the characterization but is a consequence of \rr-good.
 
 \begin{prop} \label{char} Let $R$ be a reduced ring.  Then, $R$ is \rr-good if and only if (i)~all the Pierce stalks of $R$ are \rr-good and (ii)~for $a,b\in R$ and $x\in \spec \B(R)$ if $c\in R$ is such that $c_x=a_x\wrr b_x$ then there is a neighbourhood $N$ of $x$ such that for all $y\in N$, $c_y=a_y\wrr b_y$.  \end{prop}
 
 \begin{proof} If $R$ is \rr-good then for $a,b\in R$ and $c=a\wrr b$, $c_x=a_x\wrr b_x$ globally by Proposition~\ref{wrr}. In the other direction, the usual Pierce method using the compactness of $\spec \B(R)$ works here. \end{proof}
 
 The conditions (i) and (ii) above are the same as saying that for $a,b\in R$ and each $x\in \spec\B(R)$ there is $e(x)\in \B(R)\setminus x$ such that $e(x)a \wrr e(x)b$ exists in $R$. \\
 
\setcounter{section}{6} \setcounter{thm}{0}
\noindent\tbf{6. Generalized power series rings which are wB.} If $R$ is an \rr-good ring it seems difficult to determine whether or not $R[X]$ is also \rr-good. There are too many equations involved. However, there are a few cases where there are results to be gleaned. In particular where the ring $R$ is wB. It will not be necessary to restrict the discussion to polynomial rings. 

In \cite[Section~4]{R}, P. Ribenboim defined a very general form of power series ring.  The starting point is a ring $R$ and an ordered monoid $(M,\le)$. Unlike in \cite{R}, $M$ will be written in multiplicative notation. In what follows the order is taken to be total and \emph{strict}, i.e., if $a,b,c\in M$ and $a<b$ then $ac< bc$ and $ca< cb$. Neither $R$ nor $M$ is assumed to be commutative. The 1 in $R$ and that in $M$ will both be written 1. The starting point is $R^M$, the set of functions $M\to R$. The power series ring will be defined on the subset $S=[[R^{M\le}]]$ of those elements $f$ of $R^M$ where $\supp f =\{m\in M\mid f(m)\ne 0\}$ is well-ordered. The ring structure on $S$ is defined as follows: the addition is function addition and the multiplication is convolution. The multiplication is well defined because of the restriction on the supports of elements of $S$. (See \cite{R} for details.) 

In what follows, for $f\in S$ and $m\in M$, $f(m)$ is written $f_m$. The ring $R$ embeds in $S$. For $r\in R$ define $\tilde{r}\in S$ by $\tilde{r}_m=r$ if $m=1$ and $\tilde{r}_m=0$, otherwise. In the sequel $\tilde{r}$ will just be written $r$. Notice that the centre of $R$ is in the centre of $S$. For $0\ne f\in S$, there is a least element in $\supp f$, it is called $\mu(f)$. 

With the notation above, two lemmas will be useful.

 \begin{lem}\label{ps-image}   Let $I$ be an ideal of $R$ and $\bar{R}= R/I$. Then the ring $\bar{S}$ defined by, for $f\in S$, $\bar{f}_m= \ov{f_m}$ gives a surjective ring homomorphism, $[[R^{M\le}]]\to [[\bar{R}^{M\le}]]$, with kernel $\{f\in S\mid f_m\in I, \forall m\in M\}$. \end{lem}
\begin{proof} This is clear. \end{proof}

\begin{lem} \label{ps-dom}  (1)~In the construction of $S$, assume that $R$ is a domain. Then $S$ is a domain. (2)~In the construction of $S$, assume that $R$ is reduced. Then $S$ is reduced. \end{lem}

\begin{proof} (1)~Suppose $f,g\in S$ are non-zero but $fg=0$. Put $\mu(f)= m_0$ and $\mu(g) = n_0$. Consider $0=(fg)_{m_0n_0}= \sum_{kl=m_0n_0} f_kg_l$. By the order relation, if $k<m_0$ or $l<n_0$ then $f_kg_l=0$. To have a non-zero term it would be necessary to have $m_0\le k$ and $n_0\le l$. But if either inequality is strict, $kl> m_0n_0$. Hence, $(fg)_{m_0n_0}= f_{m_0}g_{n_0} =0$. Since $R$ is a domain, this is impossible.   

(2)~Suppose that $0\ne f\in S$ with $f^2=0$. Put $\mu(f) = m_0$.  Consider $(f^2)_{m_0^2} = \sum_{kl=m_0^2}f_kf_l$. If $k<m_0$ or $l<m_0$ then $f_kf_l=0$. It follows that $(f_{m_0})^2=0$, which is impossible. Hence, $f=0$.    \end{proof}

A boolean algebra $B$ is said to be $\kappa$-complete, $\kappa$ a cardinal number, if every $E\sbq B$ with $|E| \le \kappa$ has an infimum, $\bigwedge_{e\in E}\,e$.

The following theorem was proved, in the commutative case, by Liu and Ahsan, \cite[Theorem~2.3]{LA}, i.e., where both $R$ and $M$ are commutative.  In \cite{LA}, the term \emph{PP-ring} is used for wB ring.  In what follows, both commutativity requirements are removed. Moreover, the proof uses different techniques. An earlier result of Fraser and Nicholson (see \cite{FN}) does not require $R$ to be commutative but deals with polynomial rings and the usual power series rings.

\begin{thm} \label{ps-thm}  Let $R$ be a reduced ring and $(M\le )$ a monoid which is totally ordered and in which the ordering is strict.  Let $\kappa$ be the largest cardinality of a well-ordered subset of $M$.  The ring $S = [[R^{M\le}]]$ is wB if and only if $R$ is wB and $\B(R)$ is $\kappa$-complete.    \end{thm} 

\begin{proof}  The first step is to show that $\B(S) = \B(R)$. Let $0\ne f= f^2\in S$ and put $\mu(f) = m_0$. It will be shown that $m_0=1$ and that $f= f_1\in \B(R)$. Since $(f^2)_{m_0} = \sum_{kl=m_0}f_kf_l = f_1f_{m_0} + f_{m_0}f_1= f_{m_0}$, there are four possibilities: (1)~$1< m_0$, (2)~$m_0<1$ and  $1\notin \supp f$, (3)~$m_0<1$ and $1\in \supp f$ and (4)~$m_0=1$.  

The first is impossible since, here, $f_1=0$. The second is excluded for the same reason.  Assume $m_0<1 $ and that $1\in \supp f$. Here $m_0^2< m_0 <1$. Hence, $f_{m_0^2} =0$. Also $$0=f_{m_0^2} = (f^2)_{m_0^2} = \sum_{kl=m_0^2} f_kf_l = (f_{m_0})^2 + f_1f_{m_0^2} + f_{m_0^2}f_1= (f_{m_0})^2 \;. $$  It follows that $f_{m_0}=0$, which is impossible. Hence, $m_0=1$. Thus, $f_1=f_1^2\in \B(R)$. 

Suppose that $\supp f \ne \{1\}$ and that the minimal element of $\supp f\setminus \{1\}$ is $m_1$. Then $$f_{m_1} = (f^2)_{m_1} = \sum_{kl=m_1} f_kf_l = f_1f_{m_1} + f_{m_1}f_1\;.$$ Since $f_1\in \B(R)$, $f_{m_1} = 2f_1f_{m_1} $. Since $f^3=f$, a similar computation shows that $f_{m_1} = 3f_1f_{m_1}$, giving that $f_1f_{m_1}=0$.  This shows that $f_{m_1} =0$, a contradiction. Hence, $f=f_1$ is an element of $R$. It follows that $\B(S) = \B(R)$. 

Recall that a ring $T$ is awB if and only if its Pierce stalks are domains. (This is \cite[Proposition~2.5]{BR3}, a non-commutative version of \cite[Theorem~2.2]{NR}.) 

First assume that $R$ is wB and that $\B(R)$ is $\kappa$-complete. Then, for $x\in \spec \B(S)= \spec \B(R)$, $S_x$ is a domain by Lemma~\ref{ps-dom}~(1). This shows that $S$ is awB. This means that for $f\in S$, $\ann_S f=ES$, where $E\sbq \B(R)$. For each $m\in M$, let $\ann_R f_m= e_mR$, where $e_m\in \B(R)$. Put $F= \{e_m\}_{m\in M}$. For each $e\in E$, $e$ is a lower bound for $F$ since $e = e_me$, for each $m\in M$. Hence, $F$ has lower bounds. By the assumption, $\bigwedge_F e_m= \varep$ exists. For each $e\in E$ and each $e_m$, $e\le \varep \le e_m$. Thus $\ann_S f= \varep S$. 

Next, if $R$ is not wB, there is $r\in R$ with $\ann_R r$ is not generated by an idempotent. Then $\ann_S r$ is also not generated by an idempotent. Now suppose $R$ is wB but that $\B(R)$ is not $\kappa$-complete. Then there exists $E\sbq \B(R)$, $|E| = \kappa$, with non-zero lower bounds but no greatest lower bound. Let $K$ be a well-ordered subset of $M$ with $|K|= \kappa$. Define a bijection $\phi \colon K\to E$ and set $f\in S$ to be, for $m\in M$, $f_m= 1-\phi(m)$ if $m\in K$ and $f_m=0$ otherwise.  For any lower bound $e$ of $E$, $ef =0$ since, for $m\in K$, $(ef)_m =e(1-\phi(m)) =0$. Now suppose that $\varep \in \B(R)$ is such that $\ann f=\varep S$. Then, for $m\in K$, $\varep (1-\phi(m)) =0$, making $\varep$ a lower bound for $E$. Moreover, since each $e\in E$ is in $\ann f$, $e= e\varep$. This would make $\varep$ the infimum of $E$, a contradiction. 
 \end{proof}
 
 The argument used about the element $f$ at the end of the proof above is modelled on that of \cite[Theorem~2.3]{LA}. 
 
 The ring $S$, as above, has an important subring, $[R^{M\le}]$ of those elements of $S$ of finite support. Here the theorem needs no constraints on $\B(R)$.
 
 \begin{cor} \label{ps-poly}  Let $R$ and $M$ be as in the theorem. The ring $[R^{M\le}]$ is wB if and only if $R$ is wB. \end{cor}
 
 \begin{proof} For any $n\in \N$, any boolean algebra is $n$-complete. The proof of the theorem can proceed without any assumption about $\B(R)$.  \end{proof}
 
 In the more traditional setup of a finite set of variables $ \{X_1, \dots, X_n\}$ $ = \mc{X}$ and a ring $R$, the free monoid $\mc{X}^*$  with the well-ordering based on an order of the variables, the rings $[[R^{\mc{X}^*\le}]]$ and $[R^{\mc{X}^*\le}]$ are the traditional formal power series and the monoid ring, respectively.  In the case of the free commutative monoid $\mc{X}^{*c} $, the rings become the formal power series ring in commuting variables and the ring of polynomials. In the cases of formal power series rings in a finite number of variables, $\B(R)$ needs to be $\aleph_0$-complete in order to have wB rings. 
 
 The ring $S$ has other subrings when $M$ is infinite, in particular the subring $[R^{M\le^c}]$ of those elements of $S$ of countable support. There is one type of example which is of interest.

A topological space $Y$ is called a P-space if its ring of continuous real-valued functions, $\C(Y)$ is von-Neumann regular (see \cite[4JKL]{GJ} for much more detail about these spaces).

  \begin{cor}If $Y$ is a P-space, $R=C(Y)$ and $M$ as in Theorem~\ref{ps-thm} then $[R^{M\le^c}]$ is wB. \end{cor}
   
  \begin{proof} In a P-space, cozero sets are clopen, as is a countable intersection of cozero sets (\cite[4J,(4)]{GJ}). Hence $\B(C(Y))$ is countably complete. \end{proof}

\end{document}